\documentclass[11pt]{article}
\renewenvironment{abstract}
 {\small
  \begin{center}
  \bfseries \abstractname\vspace{-.5em}\vspace{0pt}
  \end{center}
  \list{}{
    \setlength{\leftmargin}{2cm}%
    \setlength{\rightmargin}{\leftmargin}%
  }%
  \item\relax}
 {\endlist}

\usepackage{appendix}
\usepackage{amsmath, amssymb, latexsym, amsfonts, amscd, amsthm, color, cite}
\usepackage{graphicx}
\usepackage[OT2,T1]{fontenc}
\DeclareSymbolFont{cyrletters}{OT2}{wncyr}{m}{n}
\DeclareMathSymbol{\Sha}{\mathalpha}{cyrletters}{"58}
\newtheorem{theorem}{Theorem}
\newtheorem{proposition}{Proposition}
\newtheorem{lemma}{Lemma}
\newtheorem{corollary}{Corollary}
\newtheorem{definition}{Definition}
\newtheorem{remark}{Remark}
\newtheorem{example}{Example}
\newtheorem{conjecture}{Conjecture}
\newtheorem{question}{Question}

\DeclareMathOperator{\T}{\mathbb{T}}
\DeclareMathOperator{\C}{\mathbb{C}}
\DeclareMathOperator{\R}{\mathbb{R}}
\DeclareMathOperator{\Z}{\mathbb{Z}}

\newcommand{\rpunc}{R_{\Lambda}}
\newcommand{\trans}{\Omega_{trans}}

\renewcommand{\phi}{\varphi}

\renewcommand{\L}{\Lambda}

\title{\LARGE Gabor Frames for Quasicrystals, $K$-theory, and Twisted Gap Labeling}
\author{\Large Michael Kreisel \\ \normalsize Department of Mathematics \\ \normalsize University of Maryland \\ \normalsize michael.c.kreisel@gmail.com}

\date{\Large November 2014}
\begin{document}
\maketitle

\begin{abstract}
We study the connection between Gabor frames for quasicrystals, the topology of the hull of a quasicrystal $\Lambda,$ and the $K$-theory of the twisted groupoid $C^*$-algebra $\mathcal{A}_\sigma$ arising from a quasicrystal. In particular, we construct a finitely generated projective module $\mathcal{H}_\L$ over $\mathcal{A}_\sigma$ related to time-frequency analysis, and any multiwindow Gabor frame for $\Lambda$ can be used to construct an idempotent in $M_N(\mathcal{A}_\sigma)$ representing $\mathcal{H}_\L$ in $K_0(\mathcal{A}_\sigma).$ We show for lattice subsets  in dimension two, this element corresponds to the Bott element in $K_0(\mathcal{A}_\sigma),$ allowing us to prove a twisted version of Bellissard's gap labeling theorem.
\end{abstract}

%\tableofcontents

\begin{section}{Introduction}
\label{intro}
The first examples of mathematical quasicrystals were studied by Meyer in \cite{Mey}. Meyer thought of quasicrystals as generalizations of lattices which retained enough lattice-like structure to be useful for studying sampling problems in harmonic analysis. In another direction, the mathematical theory of quasicrystals began developing rapidly after real, physical quasicrystals were discovered by Shechtman et. al. \cite{Shecht}. This led to the study of the topological dynamics of the hull $\Omega_\L$ of a quasicrystal $\L,$ which are directly related to a variety of questions and constructions in symbolic dynamics (see \cite{BaaGrim} and \cite{Sadun} for an introduction). Bellissard's gap labeling conjecture provides a clear connection between the mathematics and physics \cite{BHZ}. While Bellissard's work demonstrates the value of topology and dynamics in studying the physics of quasicrystals, little has been done to integrate Meyer's original vision into this picture. The goal of the present paper is to show one avenue by which these strands of research can be connected. Namely, we will show how Gabor frames for a quasicrystal can be made compatible with its topological dynamics, and we use this connection to prove a twisted version of Bellissard's gap labeling conjecture for two-dimensional quasicrystals.

To elaborate, we will begin by describing Bellissard's gap labeling conjecture in detail. Given a quasicrystal $\Lambda,$ we can imagine a material with an electron at each point in $\Lambda.$ In order to analyze electron interactions in $\L,$ one studies a Schrodinger operator of the form
$$H_\Lambda = \frac{1}{2m}\left(\frac{\hbar}{i}\vec{\nabla} - e\vec{A}\right)^2 + \sum_{y \in \Lambda} v(\cdot - y) $$
acting on $L^2(\R^d),$ where $v$ is a suitable potential(\cite{Bel} Section 2.7). The vector potential $\vec{A}$ models the effect of a constant, uniform magnetic field. With appropriate boundary conditions, it is possible to restrict $H_\Lambda$ to an operator $H_{\Lambda, R}$ on $L^2(C_R(0))$ where $C_R(0)$ is the closed cube of side length $R.$ Then we can define the integrated density of states (IDOS)
$$\mathcal{N}(E) = \lim_{R \to \infty} \frac{1}{|C_R(0)|}| \{E' \in \text{Sp}(H_{\Lambda, R}) \, | \, E' \leq E\}|$$
which is used to express thermodynamical properties such as the heat capacity.

The IDOS can also be expressed using the language of operator algebras. There are natural $C^*$ and Von Neumann (VN) algebras related to $H_\Lambda$ which are generated by the resolvent of $H_\Lambda.$ Essentially, these operator algebras are the same as the twisted groupoid algebras $\mathcal{A}_\theta = C^*(R_\L, \theta)$ described in Section \ref{groupoidalgebra}, where the cocycle $\theta$ is determined by the magnetic field and is the restriction of a cocycle on $\R^{2d}$ (see \cite{Bel} and \cite{BHZ} for details). The algebra $\mathcal{A}_\theta$ is simple and has a unique normalized trace $Tr.$ The associated VN algebra will be denoted by $\mathcal{A}_\theta''.$ The spectral projection of $H_\Lambda$ onto $(-\infty, E]$ is denoted by $\chi_E(H_\Lambda),$ and lies in $\mathcal{A}_\theta''.$ This allows us to describe the IDOS using the trace on $\mathcal{A}_\theta''$ as
$$\mathcal{N}(E) = Tr(\chi_E(H_\Lambda))$$
which is known as Shubin's formula (\cite{Bel} Section 2.7). When $E$ lies in a spectral gap, $\chi_E(H_\Lambda)$ lies in the $C^*$-algebra $\mathcal{A}_\theta.$ In this case, the value of the IDOS is constant over the gap and can be described using the trace on $\mathcal{A}_\theta.$

Thus there is physical interest in computing the image of the trace map  
$$Tr_* \, : \, K_0(\mathcal{A}_\theta) \rightarrow \R$$
which we will call the \textbf{gap labeling group}. Moreover, from physical considerations we would expect that the gap labels can be computed from only the structure of $\Lambda$ and $\theta.$ A large part of the gap labeling group can be computed by looking at the structure of $\mathcal{A}_\theta$ as a groupoid algebra. Since the unit space of $R_\L$ is a Cantor set, any clopen set of the unit space gives a projection in $\mathcal{A}_\theta.$ The trace of the corresponding projection is simply the measure of the clopen set, which is given by a patch frequency as described in Section \ref{topology}. This leads to Bellissard's gap labeling conjecture:

\begin{conjecture}[\cite{BHZ} Problem 1.15]
\label{gapconj}
When the magnetic field $\theta=0,$ the set of gap labels is given by
$$Tr_*(K_0(\mathcal{A}_{\theta=0})) = \int_{\Omega_{trans}} C(\Omega_{trans}, \Z),$$
which is precisely the group generated by the patch frequencies of $\Lambda.$
\end{conjecture}

\noindent There are many proofs of Conjecture \ref{gapconj} in low dimensions and other special cases, and there are at least three proofs of the conjecture in full generality \cite{BBG}, \cite{BenOyo}, \cite{KP}. However, all three of these papers depend upon the results of \cite{FH}, which have been shown to be incorrect. Thus the conjecture appears to remain open in its full generality, although all cases of physical interest have been settled.

Despite the success of Bellissard's gap labeling program when $\theta=0,$ nothing seems to be known when a magnetic field is present. Part of the problem is that by Conjecture \ref{gapconj} we can ignore all parts of $K_0(\mathcal{A}_{\theta=0})$ which do not come from projections in the unit space of $R_\L$ when computing the gap labels. However, once we twist by $\theta$ the other summands in $K_0(\mathcal{A}_\theta)$ may contribute to the gap labeling group. Additionally, the methods of \cite{BBG}, \cite{BenOyo}, and \cite{KP} all apply some version of transverse index theory. This allows them to prove Conjecture \ref{gapconj} without knowing how to construct the classes in $K_0(\mathcal{A}_\theta).$ Thus one might ask: 

\begin{question}
\label{construct}
How can we construct the classes in $K_0(\mathcal{A}_\theta)$ which do not come from projections on the unit space of $R_\L?$
\end{question}

\noindent If we could construct these elements directly then we would immediately be able to compute the gap labeling group.

Motivated by Question \ref{construct}, we look at the simpler case when $\Lambda$ is a lattice. In this case, the algebra $\mathcal{A}_\theta$ is a noncommutative torus. In \cite{Rief}, Rieffel describes a general procedure for constructing all modules over noncommutative tori. In \cite{Luef1} and \cite{Luef2} Luef shows how Rieffel's construction is related to Gabor analysis. In particular, he shows how Gabor frames for lattices can be used to construct idempotents in noncommutative tori which represent Rieffel's modules. In order to prove that his modules are finitely generated and projective, Rieffel uses arguments that rely heavily on the group structure of a lattice. The construction of Luef's idempotents is more flexible since Gabor frames can be defined for any point set and not only for lattices. One roadblock to generalizing Luef's results to the setting of quasicrystals is that lattice Gabor frames are understood much better than non-uniform frames. The recent results of Gr{\"o}chenig, Ortega-Cerda, and Romero in \cite{GOCR} have greatly increased our understanding of non-uniform frames and they comprise the main technical results that we need. 

The main goal of the current paper is to adapt Rieffel's construction to the setting of quasicrystals. In Section \ref{prelim} we review the background on quasicrystals and time-frequency analysis needed to understand our main results. We also give a review of Rieffel's and Luef's work on noncommutative tori to motivate our constructions. In Section \ref{hlambda} we construct an $\mathcal{A}_\sigma$ module $\mathcal{H}_\L$ which is a representation of $\mathcal{A}_\sigma$ by time-frequency shifts. In order to show that $\mathcal{H}_\L$ is finitely generated and projective, we will need the following theorem, which is interesting in its own right in the context of Gabor analysis:

\begin{theorem}
\label{frameexistence}
Let $\Lambda \subset \R^{2d}$ be a quasicrystal. Then there exist functions $g_1, \dots, g_N$ so that for any $T \in \Omega_\L, g_1, \dots g_N$ generate a multiwindow Gabor frame for $L^2(\R^d)$ and an $M^p$-frame for all $1 \leq p \leq \infty.$
\end{theorem}

In Section \ref{frameproofs} we show a number of ways in which Gabor frames for a quasicrystal are compatible with the topology of the hull $\Omega_\L,$ including a proof of Theorem \ref{frameexistence}. We then use Theorem \ref{frameexistence} to show that $\mathcal{H}_\L$ is finitely generated and projective. In order to compute the dimension of $\mathcal{H}_\L,$ we apply a deep result from \cite{BCHL2} on the frame measure for non-uniform Gabor frames. This yields the following theorem:

\begin{theorem}
\label{fingenproj}
The module $\mathcal{H}_\L$ is finitely generated and projective as an $\mathcal{A}_\sigma$-module, and thus defines a class $[\mathcal{H}_\L] \in K_0(\mathcal{A}_\sigma).$ The dimension of $\mathcal{H}_\L$ is given by 
$$Tr([\mathcal{H}_\L]) = \frac{1}{\text{Dens}(\L)},$$
so that $\frac{1}{\text{Dens}(\L)}$ lies in the gap labeling group of $\mathcal{A}_\sigma.$
\end{theorem}

Theorem \ref{fingenproj} indicates that when the cocycle $\theta$ is nontrivial, the gap labeling group may be generated by more than just the patch frequencies. Note that while the density of $\L$ is intrinsic to the space $\Omega_\L,$ it is not an isomorphism invariant of the groupoid $R_\L.$ For example, we can apply a linear map $A$ with $\text{det}(A) \neq 1$ to $\Lambda.$ The groupoids $R_\L$ and $R_{A\L}$ are isomorphic, however $\text{Dens}(A\L) = \frac{\text{Dens}(\L)}{|\text{det}(A)|}.$ If the cocycle $\sigma$ is preserved by the isomorphism then $\mathcal{H}_\L$ and $\mathcal{H}_{A\L}$ are both modules over $\mathcal{A}_\sigma,$ but represent different elements in $K_0(\mathcal{A}_\sigma).$ Thus by deforming $\L$ in a way which does not alter the groupoid $R_\L$ or the cocycle $\sigma,$ we can construct many modules over $\mathcal{A}_\sigma$ using our methods.

In Section \ref{KGap}, we illustrate this point by computing the gap labeling group for any standard cocycle $\theta$ when $\Lambda \subset \R^2$ is contained in a lattice. This involves showing exactly how the modules $\mathcal{H}_{A\L}$ fit into $K_0(\mathcal{A}_\theta),$ which can be computed easily using the Pimsner-Voiculescu exact sequence. In higher dimensions $K_0(\mathcal{A}_\theta)$ is larger and the modules $\mathcal{H}_{A\L}$ are not enough to generate the rest of $K_0(\mathcal{A}_\theta),$ although they do give us some information about the gap labeling group. We are able to prove the following theorem which is a partial generalization of our results to higher dimensions. Let $\L \subset \R^{d}$ be a marked lattice with an aperiodic coloring satisfying the definition of a quasicrystal. Then $\Omega_\L$ naturally has the structure of a fiber bundle 
$$p: \Omega_\L \rightarrow \T^{d}$$
over the torus $\T^{d}$ with the Cantor set as its fibers. The fiber bundle structure comes from viewing $\Omega_\L$ as the suspension of the Cantor set $\Omega_{trans}$ by an action of $\Z^{d}.$

\begin{theorem}
\label{injective}
The induced map $p^*:K^0(\T^{d}) \rightarrow K^0(\Omega_\L)$ is injective. Furthermore, we can compare the image of $p^*$ with the image of $r_*,$
$$r_*: K_0(C(\Omega_{trans})) \rightarrow K_0(C(\Omega_{trans}) \rtimes \Z^{d}) \cong K_0(C(\Omega_\L) \rtimes \R^{d}) \cong K^0(\Omega_\L)$$
where $r_*$ is induced by the inclusion $r: C(\Omega_{trans}) \rightarrow C(\Omega_{trans}) \rtimes \Z^{d}.$ The intersection of the images of $p^*$ and $r_*$ is generated by $[1],$ the class of the trivial bundle.
\end{theorem}

\noindent The proof of Theorem \ref{injective} shows that it can be useful to study the twisted algebras $\mathcal{A}_\sigma$ even if one's primary goal is to understand the topology of $\Omega_\L.$

\subsection*{Acknowledgements}
I would like to thank my thesis adviser, Jonathan Rosenberg, for constant support and encouragement during this research. I would also like to thank Scott Schmieding for introducing me to quasicrystals. Without his help I would never have been able to pursue this line of research. Finally, I would like to thank Franz Luef, Antoine Julien, and the mathematics department at NTNU for inviting me to present this research while it was underway, and for many helpful discussions.
\end{section}

\begin{section}{Preliminaries}
\label{prelim}
\begin{subsection}{Topology of Quasicrystals}
\label{topology}
The main objects of our investigation are quasicrystals, so we begin with a review of the topological and dynamical properties of a quasicrystal, as well as properties of the associated operator algebras. We will state the basic definitions and theorems for even dimensional quasicrystals since it will be simplify notation later, however the same definitions and theorems apply in any dimension. We will always think of $\R^{2d} \cong \R^d \times \hat{\R}^d$ as time-frequency space, and elements $z \in \R^{2d}$ will be written as $z=(x, \omega)$ when it is necessary to emphasize this point of view. 

\begin{definition}
Let $\Lambda \subset \R^{2d}$ be a discrete set.
\begin{enumerate}
\item The \textbf{hole} of $\Lambda$ is defined to be
$$ \rho(\Lambda) := \sup_{z \in \R^{2d}}\inf_{\lambda \in \Lambda}|z-\lambda|$$
and $\Lambda$ is called \textbf{relatively dense} if $\rho(\Lambda) < \infty.$
\item $\Lambda$ is called \textbf{relatively separated} if
$$\text{rel}(\Lambda) := \sup\{\#(\Lambda \cap C_1(z)) : z \in \R^{2d}\} < \infty$$
where $C_1(z)$ is the cube of side length 1 centered at $z.$
\item $\Lambda$ is called \textbf{uniformly discrete} if there is an open ball $B_r(0)$ s.t. $(\Lambda-\Lambda)\cap B_r(0)=\{0\}.$
\end{enumerate}
If $\Lambda$ is both relatively dense and uniformly discrete then it is called a \textbf{Delone set}. A Delone set $\Lambda$ is called \textbf{aperiodic} if $\Lambda - z \neq \Lambda$ for any $z \in \R^{2d}.$
\end{definition}
\noindent In Gabor analysis, the goal is to recover a function from samples of its Short Time Fourier Transform on a discrete set (see Section \ref{timefreq}). Often the sampling set is assumed to be a lattice, however there are now a variety of results available which treat sampling on non-uniform sets as well (\cite{BCHL2}, \cite{GOCR}).    While these results are able to deal with sampling on arbitrary Delone sets, we will restrict our attention to quasicrystals, which are Delone sets with additional regularity properties.

\begin{definition}
Let $\Lambda$ be a Delone set. The sets $B_r(z)\cap \Lambda$ where $z \in \Lambda$ are called the $\mathbf{r}$\textbf{-patches} of $\Lambda.$
\begin{enumerate}
\item If for any fixed $r$ there are only finitely many $r$-patches up to translation, then $\Lambda$ is said to be of \textbf{finite local complexity (FLC).}
\item For an $r$-patch $P$ and a set $A\subset \R^{2d}$ we define
$$L_P(A)= \#\{z \in \R^{2d}\;|\;P - z \subset A\}.$$
Thus $L_P(A)$ counts the number of times $P$ appears in $A.$ For a sequence of balls $B_{r_k}(z)$ in $\R^{2d}$ such that $r_k$ goes to $\infty,$ we define the \textbf{patch frequency} of $P$ to be
$$freq(P, \Lambda)=\lim_{k \to \infty} \frac{L_P(B_{r_k}-z)}{vol(B_{r_k})}$$
if this limit exists uniformly in $z$ and independent of the choice of balls $B_{r_k}.$ If the patch frequencies exist for all patches $P \subset \Lambda$ then $\Lambda$ is said to have \textbf{uniform cluster frequencies (UCF)}. 
\end{enumerate}
A Delone set is called a \textbf{quasicrystal} if it is FLC and has UCF. 
\end{definition}

\begin{example}[Model sets]
Consider the space $\R^{2d} \times G,$ where $G$ is a locally compact abelian group. Let $\pi_1$ and $\pi_2$ be the canonical projections onto $\R^{2d}$ and $G$ respectively. Fix $D \subset \R^{2d} \times G$ a discrete cocompact subgroup and $W \subset G$ a relatively compact subset whose boundary has Haar measure 0. Also assume that $\pi_2(D)$ is dense in $W.$ We define the \textbf{model set} or \textbf{cut and project set} $\Lambda_W$ by
$$\Lambda_W = \{\pi_1(d)\;|\; d \in D, \pi_2(d) \in W\}.$$
Any cut and project set (except for a lattice) is aperiodic, FLC, and has UCF (\cite{BHZ}, \cite{LMS}).
\end{example}

The study of quasicrystals has to a large extent been driven by the study of electron interactions in aperiodic solids. Given an electron in an aperiodic solid, we might assume that it will only interact with nearby electrons since the forces drop off rapidly as distances increase. Thus for the study of electron interactions, it is natural to treat two quasicrystals $\Lambda$ and $\Lambda'$ as the same if they contain precisely the same local patterns. One could formalize this by saying that any $r$-patch appearing in $\Lambda$ also appears as an $r$-patch in $\Lambda'$ and vice versa, and in this case we say $\Lambda$ and $\Lambda'$ are \textbf{locally isomorphic}. For example, any translate $\Lambda - z$ is clearly locally isomorphic to $\Lambda.$ The collection of all quasicrystals which are locally isomorphic to $\Lambda$ will be called the $\textbf{hull}$ of $\Lambda$ (denoted $\Omega_\Lambda),$ and this object is useful in studying the physics of aperiodic solids (see \cite{BHZ}).

Now we present another construction of the hull which demonstrates how $\Omega_\L$ can be given the structure of a topological dynamical system. Given two Delone sets $\Lambda, \Lambda',$ define
$$R(\Lambda, \Lambda')=\sup \{r \;|\; \exists z\in \R^{2d} \;\text{with}\; ||z||<\frac{1}{r}, B_r\cap (\Lambda - z) = \Lambda' \cap B_r\}.$$
We can define the distance between $\Lambda$ and $\Lambda'$ as
$$d(\Lambda, \Lambda')=\min \left\{1, \frac{1}{R(\Lambda, \Lambda')}\right\}.$$
Intuitively, two Delone sets are close if they agree in a large ball around the origin after a small translation. This defines a metric $d$ on the space of all Delone subsets of $\R^{2d}.$

\begin{definition}
Given a Delone set $\Lambda,$ the \textbf{orbit} of $\Lambda$ is $O_\Lambda = \{\Lambda - z\;|\; z\in \R^{2d}\}.$ The \textbf{hull} $\Omega_\Lambda$ is the closure of $O_\Lambda$ in the metric $d.$
\end{definition}
\noindent The hull $\Omega_\Lambda$ comes with a natural action of $\R^{2d}$ by translation. The following proposition shows how regularity properties of $\Lambda$ can be translated into properties of the dynamical system $(\Omega_\Lambda, \R^{2d}):$

\begin{proposition}[\cite{BHZ}, \cite{LMS}]
Let $\Lambda$ be an aperiodic Delone set.
\begin{enumerate}
\item $\Lambda$ is FLC iff $\Omega_\Lambda$ is compact.
\item $\Lambda$ has UCF iff the dynamical system $(\Omega_\Lambda, \R^{2d})$ is minimal and uniquely ergodic.

\end{enumerate}
 
\end{proposition}

Thus we see that for any quasicrystal $\L$ we have an associated dynamical system $(\Omega_\Lambda, \R^{2d})$ which is compact, minimal, and uniquely ergodic. In fact, we have an explicit description of the ergodic measure $\mu$ using patch frequencies. Given a patch $P$ in $\Lambda,$ and $V \subset \R^{2d}$ a precompact open set, define the \textbf{cylinder set}
$$\Omega_{P,V} = \{\Lambda' \in \Omega_\Lambda\;|\; P-z \subset \Lambda'\;\text{for some} \; z \in V\}.$$
The cylinder sets form a basis for the topology on $\Omega_\Lambda,$ so it suffices to describe the ergodic measure for cylinder sets. Fix $\eta(\Lambda)$ so that any ball of radius $\eta(\Lambda)$ contains at most one point of $\Lambda.$ If $\text{diam}(V)<\eta(\Lambda),$ then the measure of $\Omega_{P,V}$ is given by
$$\mu(\Omega_{P,V})=Vol(V)freq(P,\Gamma)$$
where $\Gamma \in \Omega_\Lambda$ is any Delone set containing $P$ as a patch. Since we can also describe the hull using local isomorphism classes, for any $\Lambda' \in \Omega_\Lambda$ the quantities $\text{rel}(\Lambda')$ and $\rho(\Lambda')$ are equal to $\text{rel}(\Lambda)$ and $\rho(\Lambda)$ respectively. Thus we may think of these quantities as associated to the hull itself, and not just to a particular point set contained in it. Furthermore, the patch frequencies are also independent of the choice of point set $\Lambda' \in \Omega_\Lambda$ so that the patch frequencies and density can be associated to the tiling space as a whole as well.

While the hull $\Omega_\L$ appears naturally from physical considerations, we will consider now a different space which appears more naturally in the context of harmonic analysis. We would like to think of a quasicrystal $\L$ as a collection of shifts we can apply to a function. The shifts might simply be translations (see \cite{MaMe}), but in the case of Gabor analysis they will be time-frequency shifts. In this vein, we consider 
$$O_{trans}^\L := \{\Lambda-z\;|\; z\in \L\},$$
the collection of Delone sets which are translates of $\L$ by points in $\L.$

\begin{definition}
We define the \textbf{canonical transversal} $\Omega_{trans}$ as the closure of $O_{trans}^\L$ in the metric $d.$
\end{definition}
\noindent Note that the canonical transversal can also be defined as
$$\Omega_{trans} = \{\L' \in \Omega_\L \,|\, 0 \in \L'\},$$
and is a transversal to the action of $\R^{2d}$ on the hull. 

Topologically $\Omega_{trans}$ is a Cantor set, and it comes with a  measure which, by abuse of notation, we shall also call $\mu.$ Given a patch $P \subset \L,$ we can define 
$$\Omega_P := \{\Lambda' \in \Omega_{trans} \, | \, P - z = B_r(0) \cap \Lambda' \, \text{for some} \, r \in \R, z \in \R^{2d}\}.$$
The set $\Omega_P$ contains exactly the point sets in $\Omega_{trans}$ which have the pattern $P$ centered at the origin. The sets $\Omega_P$ form a clopen basis for the topology on $\Omega_{trans},$ and $\mu(\Omega_P) = freq(P, \Lambda).$ The hull $\Omega_\Lambda$ is locally the product of $\Omega_{trans}$ and $\R^{2d}$ as both a topological space and a measure space. However, it is not always the case that $\Omega_{trans}$ carries an action of $\Z^{2d}$ so that $\Omega_\L$ is the suspension of $\Omega_{trans}.$ This will be an important point to keep in mind during Section \ref{KGap}. 
\end{subsection}

\begin{subsection}{The Groupoid $C^*$-Algebra of $\L$}
\label{groupoidalgebra}
We are now ready to describe our main object of study: the groupoid $C^*$-algebra associated to $\Omega_{trans}.$ We consider the equivalence relation
$$R_\Lambda = \{ (T,T') \in \Omega_{trans} \times \Omega_{trans} \, | \, T \; \text{is a translate of} \; T'\}$$
and an element of $R_\L$ will be written as $(T-z, T)$ where $z \in \R^{2d}.$ We give $R_\L$ a topology by declaring that a sequence $(T_k - z_k, T_k) \rightarrow (T-z, T)$ iff $T_k \rightarrow T$ in $\Omega_{trans}$ and $|z_k - z| \rightarrow 0.$ With this topology, $R_\L$ has the structure of a locally compact, principal, r-discrete groupoid (see \cite{Ren}). The unit space of $R_\L$ is given by elements of the form $(T,T).$ We can compose two elements $(T-z, T), (T'-w, T')$ only if $T' = T-z,$ and in this case 
$$(T-z-w, T-z)*(T-z,T) = (T-z-w, T).$$

This groupoid captures the idea of shifting by exactly the points in $\Lambda.$ To see this, note that $(\Lambda, \Lambda - z) \in R_\L$ iff $z \in \L.$ Thus the orbit of $\L$ in $R_\L$ is in correspondence with the points of $\L,$ and the element $(\L, \L-z)$ can be thought of as a shift by $z.$ This will be made more explicit in Section \ref{hlambda}, where we will construct a projective representation of $R_\L$ using time-frequency shifts. Anticipating this, we will describe  the cocycle on $R_\L$ which will be involved in this projective representation. First, let $\theta$ be a 2-cocycle on $\R^{2d}.$ We can use $\theta$ to construct a 2-cocycle on $R_\L,$ denoted $\theta_\L,$ using the formula
$$\theta_\L\left((T-z,T), (T'-w,T')\right) = \theta(z,w).$$
Cocycles of this form will be called \textbf{standard cocycles}, and when it is clear we will drop the subscript from $\theta_\L$ and refer to both cocycles as $\theta.$ We will be particularly concerned with the symplectic cocycle $\sigma$ on $\R^{2d}$ given by
$$\sigma(z,w) = e^{-2 \pi i x \omega'}$$
where $z=(x,\omega)$ and $w=(x',\omega').$ 

Following \cite{Ren} and \cite{BHZ}, we construct a $C^*$-algebra from $R_\L$ and a 2-cocycle $\theta.$ To construct the $C^*$-algebra $\mathcal{A}_\theta = C^*(R_\L, \theta),$ we begin with $C_c(R_\L),$ with the product
$$f*g(T-z,T) :=\sum_{w \in T} f(T-z, T-w)g(T - w,T)\theta\left((T-z,T-w),(T-w,T)\right)$$
and involution defined by
$$f^*(T-z, T):=\overline{f(T,T-z)}\theta\left((T-z,T), (T, T-z)\right)$$
where $z=(x,\omega).$ For the symplectic cocycle $\sigma,$ the multiplication and involution can be written as
$$f*g(T-z,T) :=\sum_{w=(x',\omega') \in T} f(T-z, T-w)g(T - w,T)e^{-2 \pi i\omega'(x-x')}$$
and 
$$f^*(T-z, T):=\overline{f(T,T-z)}e^{-2 \pi i x\omega}$$
respectively. We can define a norm on $C_c(R_\L)$ by taking the sup over all the norms coming from the bounded representations of $C_c(R_\L)$ (see \cite{Ren} Chapter 2 for details, or \cite{BHZ} Section 4.1 for a description specific to quasicrystals). After completing $C_c(R_\L)$ in this norm, we obtain the $C^*$-algebra $\mathcal{A}_\theta.$

For a standard cocycle $\theta,$ we can also construct a cocycle on the action groupoid $C(\Omega_\L) \rtimes \R^{2d},$ and in this case $\mathcal{A}_\theta$ is Morita equivalent to the twisted crossed product $C^*(C(\Omega_\L) \rtimes \R^{2d}, \theta)$ \cite{MRW}. Since the action of $\R^{2d}$ on $\Omega_\L$ is minimal and uniquely ergodic, both algebras are simple and have a unique normalized trace given by integrating over the unit space of their respective groupoids. For a function $f \in C_c(\Omega_\L),$ the trace is given by
$$Tr(f) = \int_{\Omega_{trans}}f(T,T) dT,$$
and after applying Birkhoff's ergodic theorem we can write
$$Tr(f) = \lim_{k \to \infty}\frac{1}{|\Lambda \cap C_k|}  \sum_{z \in (\Lambda \cap C_k)} f(T-z,T-z)$$
so that the trace is expressed as an average over the values of $f$ on the orbit of $\Lambda.$

For a standard cocycle $\theta$ we can compute the $K$-theory of $\mathcal{A}_\theta$ by appealing to the following theorem of Gillaspy \cite{Gil}:

\begin{theorem}[\cite{Gil} Thm. 5.1]
\label{Gillaspy}
Let $G$ be a second countable locally compact Hausdorff group acting on a second countable locally compact Hausdorff space $X$ such that $G$ satisfies the Baum-Connes conjecture with coefficients, and let $\omega_t$ be a homotopy of continuous 2-cocycles on the transformation group $X$ $\rtimes$ G. For any $t \in [0,1],$ the $*$-homomorphism 
$$q_t: \, C^*_r(G \rtimes X \times [0,1], \omega) \rightarrow C^*_r(G \rtimes X, \omega_t),$$
given on $C_c(G \rtimes X \times [0,1])$ by evaluation at $t \in [0,1],$ induces an isomorphism 
$$K_*(C^*_r(G \rtimes X \times [0,1], \omega)) \cong K_*(C^*_r(G \rtimes X, \omega_t)).$$
\end{theorem}

\noindent Theorem \ref{Gillaspy}, combined with the Connes-Thom isomorphism and the Morita equivalence between $\mathcal{A}_\theta$ and $C^*(C(\Omega_\L) \rtimes \R^{2d}, \theta),$ gives
$$K_*(\mathcal{A}_\theta) \cong K_*(C^*(C(\Omega_\L) \rtimes \R^{2d}, \theta)) \cong K_*(C(\Omega_\L) \rtimes \R^{2d}) \cong K_*(C(\Omega_\L)) \cong K^*(\Omega_\L).$$
Theorem \ref{Gillaspy} applies since any cocycle on $\R^{2d}$ is homotopic to the trivial cocycle, essentially by the straight line homotopy. Unfortunately, the $K$-theory of $\Omega_\L$ can be quite complicated. In many cases $K^0(\Omega_\L)$ will not be finitely generated, and there are examples where it is not torsion free \cite{GHK}. Because of these complexities, it is in general difficult to see how our module $\mathcal{H}_\L$ fits into $K_0(\mathcal{A}_\sigma).$ In Section \ref{KGap}, we will show that when $\Lambda \subset \R^2$ is a subset of a lattice these difficulties can be overcome, and an understanding of how $\mathcal{H}_\L$ fits into $K_0(\mathcal{A}_\sigma)$ is enough to compute $Tr_*(K_0(\mathcal{A}_\sigma)).$

\end{subsection}

\begin{subsection}{Time-Frequency Analysis}
\label{timefreq}
In this section we review some basic concepts from time-frequency analysis, along with the recent work of Gr$\ddot{\text{o}}$chenig, Ortega-Cerda, and Romero in \cite{GOCR} which will be the technical backbone for many of our proofs.

\begin{definition}
For a point $z =(x, \omega) \in \R^{2d}$ we denote by $\pi(z)$ the \textbf{time-frequency shift} by $z,$ which operates on $L^2(\R^d)$ by
$$\pi(z)f(t) = M_\omega T_x f(t) = e^{2\pi i \omega t}f(t-x).$$
Here $M_\omega$ denotes the \textbf{modulation operator}
$$M_\omega f(t) = e^{2\pi i \omega t}f(t)$$
and $T_x$ denotes the \textbf{translation operator}
$$T_x f(t) = f(t-x).$$
Fix $g \neq 0 \in L^2(\R^d)$ which we will call the window function. Then the \textbf{Short Time Fourier Transform (STFT)} of $f \in L^2(\R^d)$ with respect to the window $g$ is
$$V_g f (x,\omega) = \int_{\R^d} f(t) \overline{g(t-x)} e^{-2 \pi i t \omega} dt \; \text{for} \; (x,\omega) \in \R^{2d}.$$
\end{definition}

\noindent The STFT of a function $f$ with respect to the window $g$ is an attempt to decompose $f$ into time-frequency shifts of $g.$ If $g$ is supported on a small set around the origin then we can view $V_gf$ as an attempt to measure the ``local frequencies'' present in $f.$ Similar to the Fourier transform, the STFT has the following continuous reconstruction formula:

\begin{proposition}[\cite{Gro}]
Fix $g, \gamma \in L^2(\R^d)$ s.t. $\langle g, \gamma \rangle \neq 0.$ Then for all $f \in L^2(\R^d),$
$$f= \frac{1}{\langle g, \gamma \rangle}\int\int_{\R^{2d}}V_g f(x,\omega)M_\omega T_x \gamma \; d\omega dx.$$
\end{proposition}

\noindent A central goal in Gabor analysis is to look for discrete versions of this reconstruction formula. This idea is expressed through the language of frames.

\begin{definition}
\label{frameinequalities}
A sequence $(e_j)_{j \in J}$ in a separable Hilbert space $\mathcal{W}$ is called a \textbf{frame} if there exist constants $A,B>0$ s.t. for all $f \in \mathcal{W}$
$$A||f||^2 \leq \sum_{j \in J} |\langle f, e_j \rangle|^2 \leq B||f||^2.$$

\end{definition}
\noindent Any frame $(e_j)$ has an associated frame operator $S$ given by
$$Sf= \sum_{j \in J} \langle f, e_j\rangle e_j,$$
which is the composition of the \textbf{analysis} and  \textbf{synthesis operators}
$$(Cf)_j = \langle f, e_j\rangle$$
$$D(\{a_j\}_{j \in J}) = \sum_{j \in J} a_j e_j.$$
We have a (non-unique, non-orthogonal) expansion of $f$ given by
$$f=\sum_{j \in J} \langle f, S^{-1}e_j\rangle e_j$$
where the elements $S^{-1}e_j$ are known as the \textbf{dual frame}. Thus if we wish to discretize the STFT, we can choose a subset $\Lambda \subset \R^{2d}$ and a window $g$ and ask whether the set 
$$\mathcal{G}(g, \L) =: \{\pi(z) g \, |\, z \in \Lambda\}$$
forms a frame for $L^2(\R^d).$ Such frames are called \textbf{Gabor frames} for $\Lambda.$ More generally, we can choose finitely many functions $g_1, \dots, g_N$ and look for \textbf{multiwindow Gabor frames} of the form 
$$\mathcal{G}(g_1, \dots, g_N, \L) := \{\pi(z) g_i \, |\, i=1 \dots, N,  z \in \Lambda\}.$$
In this case elements of the dual frame will be denoted by $\tilde{g_i}_z = S^{-1}(\pi(z)g_i).$ When $\L$ is a lattice, the dual frame will also have the structure of a Gabor frame given by $\mathcal{G}(\tilde{g_1}, \dots, \tilde{g_N}, \L)$ where $\tilde{g_i} = S^{-1}g_i.$

With this background in place, it is natural to ask:

\begin{question}
\label{question2}
Given a quasicrystal $\L,$ when can we find functions $g_1, \dots, g_N$ so that $\mathcal{G}(g_1, \dots, g_N, \L)$ is a Gabor frame for $\L?$
\end{question}

\noindent Much of the work in Gabor analysis has focused on the case where $\L$ is a lattice. However, recent results in \cite{GOCR} took a large step towards answering this question not just for quasicrystals, but for any discrete set $\L.$ In order to explain their results, it will be necessary to introduce the modulation spaces $M^p(\R^d).$

\begin{definition}
Fix a non-zero $g \in \mathcal{S}(\R^d).$ For $1 \leq p \leq \infty$ we define the \textbf{modulation spaces}
$$M^p(\R^d) := \{ f \in \mathcal{S}'(\R^d) \, | \, V_gf \in L^p(\R^{2d})\}$$
with the norm $||f||_{M^p} = ||V_gf||_p.$ 

\end{definition}

Different choices for $g$ give rise to equivalent norms on $M^p(\R^d).$ The modulation space $M^1(\R^d)$ consists of good windows for Gabor analysis. When $g \in M^1(\R^d)$ the analysis and synthesis operators for a Gabor system $\mathcal{G}(g, \L)$ are bounded between $M^p(\R^d)$ and $l^p(\L):$
$$||C_{g, \L}f||_{l^p} \leq \text{rel}(\L) ||g||_{M^1}||f||_{M^p}$$
$$||D_{g, \L}c||_{M^p} \leq \text{rel}(\L) ||g||_{M^1}||c||_{l^p}.$$
A Gabor system $\mathcal{G}(g, \L)$ with $g \in M^1(\R^d)$ will be called an $\boldsymbol M^{\boldsymbol p}$\textbf{-frame} if $C_{g,\L}$ is bounded below on $M^p(\R^d).$ This is equivalent to having constants $A,B$ so that for all $f \in M^p(\R^d)$
$$\sqrt{A}||f||_{M^p} \leq ||S_{g, \L}f||_{M^p} \leq \sqrt{B}||f||_{M^p}.$$
In this case the frame operator $S_{g, \, \L}$ is invertible on $M^p(\R^d).$ Theorem $3.2$ in \cite{GOCR} implies that when each $g_i \in M^1(\R^d)$ then $\mathcal{G}(g_1, \dots, g_N ,\L)$ is a frame for $L^2(\R^d)$ if and only if it is an $M^p$-frame for all $p.$

Now we are ready to state the result from \cite{GOCR} which gives sufficient conditions for answering Question \ref{question2}. For $g \in M^1(\R^d)$ and $\delta>0,$ we can define the $M^1$ modulus of continuity of $g$ as
$$\omega_\delta(g) = \sup_{|z-w|\leq\delta}||e^{2\pi i z_2}g(t-z_1) - e^{2\pi i w_2}g(t-w_1)||_{M^1}$$
where $z = (z_1,z_2)$ and $w = (w_1, w_2) \in \R^{2d}.$ It is clear that $\omega_\delta \rightarrow 0$ as $\delta \rightarrow 0.$

\begin{theorem}[\cite{GOCR}]
\label{gocrthm}
For $g \in M^1(\R^d)$ with $||g||_2 =1$ choose $\delta > 0$ so that $\omega_\delta(g)<1.$ If $\Lambda \subset \R^{2d}$ is relatively separated and $\rho(\Lambda)<\delta,$ then $\mathcal{G}(g, \Lambda)$ is a Gabor frame for $L^2(\R^d).$
\end{theorem}

\noindent From this result, we can see that when $\rho(\Lambda)$ is small enough there will be many windows $g$ for which $\mathcal{G}(g,\Lambda)$ is a Gabor frame. Furthermore, when $g$ is one of these admissible windows, $\mathcal{G}(g, \Lambda')$ will also form a Gabor frame for any $\Lambda' \in \Omega_\Lambda$ since $\rho(\Lambda') = \rho(\Lambda).$ However, when $\rho(\Lambda)$ is large we cannot expect $\mathcal{G}(g,\Lambda)$ to form a Gabor frame for any $g.$ In fact, the Balian-Low theorem for non-uniform frames proven in \cite{GOCR} shows that if $\mathcal{G}(g, \L)$ is a frame then $\text{Dens}(\L) > 1.$ In this case, we can only expect a multiwindow Gabor frame to exist. 

Finally, we will need to introduce one more function space needed for the proofs in Section \ref{gaborframes}. The \textbf{Wiener amalgam space} $W(L^\infty, L^1)(\R^d)$ consists of all functions $f \in L^\infty(\R^d)$ such that
$$||f||_{W(L^\infty, L^1)} := \sum_{k \in \Z^d} ||f||_{L^\infty([0,1]^d+k)}<\infty.$$
It is a standard result (see \cite{Gro} Proposition 12.1.11) that when $g \in M^1(\R^d)$ then for any $f \in M^1(\R^d), V_gf \in W(L^\infty, L^1)(\R^{2d})$ and $||V_gf||_{W(L^\infty, L^1)} \leq C||f||_{M^1}||g||_{M^1}.$ Also note that if $f \in W(L^\infty, L^1)(\R^d)$ and $T \subset \R^d$ is a Delone set then we have the inequality
\begin{align}
\label{relinequality}
\sum_{t \in T} |f(t)| \leq \text{rel}(T)||f||_{W(L^\infty, L^1)}.
\end{align}
If $T \in \Omega_\Lambda$ then the bound in this inequality is independent of $T$ since $\text{rel}(T) = \text{rel}(\Lambda).$

\end{subsection}

\begin{subsection}{Constructing Modules over Noncommutative Tori}
\label{nctori}
Next we will review Rieffel's results in \cite{Rief} on constructing modules over noncommutative tori and relate them to Gabor analysis as in \cite{Luef1}, \cite{Luef2}. 

\begin{definition}
Let $L \subset \R^{2d}$ be a lattice. The $C^*$-algebra $A_L$ generated by the time-frequency shifts $\{\pi(z) \, | \, z \in L \}$ is called a \textbf{noncommutative torus}.
\end{definition}
\noindent We can also define noncommutative tori as twisted convolution algebras. We take $l^1(L)$ with twisted convolution
$$a*b(l) := \sum_{\mu \in L} a(\mu)b(l-\mu)\sigma(\mu, l-\mu)$$
where $\sigma$ is the symplectic cocycle on $\R^{2d}.$ This is equivalent to taking the twisted group algebra $A_\theta = C^*_r(\Z^{2d}, \theta)$ where $\theta = \sigma|_L.$ The group algebra is generated by unitaries $U_{\vec{n}}$ which correspond to the Dirac $\delta$-functions at the elements of $\Z^{2d}.$ Any cocycle on $\Z^{2d}$ is given by a skew symmetric matrix $\Theta$ which describes the commutation relations between the $U_{\vec{n}}:$
$$U_{\vec{n}}U_{\vec{m}} = e^{2\pi i \vec{n} \Theta \vec{m}}U_{\vec{m}}U_{\vec{n}}.$$
When the off diagonal entries of this matrix are all irrational and rationally independent, we call the cocycle \textbf{totally irrational}. The standard trace on $A_\theta$ is given by 
$$Tr_{A_\theta}\left(\sum_{\vec{n} \in \Z^{2d}}a_{\vec{n}} U_{\vec{n}}\right) = a_0.$$
When $\theta$ is totally irrational the map $Tr_{A_\theta *} : K_0(A_\theta) \rightarrow \R$ is injective, although in general it will not be \cite{Ell}.

Each of these definitions of the noncommutative torus comes with its own advantages. By viewing a noncommutative torus as a twisted group algebra $A_\theta$ we can easily compute its $K$-theory. Any skew symmetric matrix $\Theta$ is homotopic to the zero matrix by the straight line homotopy, so Theorem \ref{Gillaspy} applies\footnote{There are many ways to compute the $K$-theory of noncommutative tori, but we use Theorem \ref{Gillaspy} since we will need this specific isomorphism later.} and shows $K_*(A_\theta) \cong K^*(\T^{2d}).$ On the other hand, when we have a lattice $L$ such that $\sigma|_L = \theta,$ the algebra $A_L \cong A_\theta$ and describes $A_\theta$ in a specific representation. Rieffel's insight was that different lattices can produce different representations of $A_\theta,$ and that these representations exhaust the classes in $K_0(A_\theta).$  

More precisely, we define the \textbf{smooth noncommutative torus} 
$$A_L^\infty := \left\{ \sum_{z \in L}a_{z} \pi(z) \in A_L \, | \, a_{z} \; \text{decays faster than any polynomial} \right\},$$
and the analogous smooth subalgebra of $A_\theta$ is defined similarly. The algebra $A_L^\infty$ is a spectrally invariant subalgebra of $A_L.$ There is a canonical action of $A_L^\infty$ on $\mathcal{S}(\R^{d})$ by time-frequency shifts, and we denote this $A_L^\infty$-module by $V_L.$ We have
$$Tr_{A_L *}([V_L]) = \text{vol}(L)= \frac{1}{\text{Dens}(L)},$$
and this last equality already suggests how the dimension of this module will generalize to quasicrystals. We identify a lattice $L$ with a linear map $A$ such that $A\Z^{2d}=L.$ If we fix a cocycle $\theta$ then $\sigma|_L = \theta$ exactly when $A^*\sigma = \theta.$

\begin{theorem}[Rieffel \cite{Rief}]
\label{Rieffel}
Fix a cocycle $\theta$ on $\Z^{2d}.$ Any invertible linear map $A$ such that $A^*\sigma = \theta$ gives rise to an $A_\theta^\infty$-module $V_{A\Z^{2d}}.$ These modules are finitely generated and projective, and any class in $K_0(\mathcal{A_\theta^\infty})$ can be represented as $[V_{A\Z^{2d}}]$ for some $A.$ 
\end{theorem}

\noindent It is possible to endow the modules $V_L$ with the structure of a Hilbert pre-$C^* A_L^\infty$-module structure, allowing us to define a norm on $V_L.$ Its completion in this norm yields a module over $A_L,$ and Theorem \ref{Rieffel} holds for $A_L$ as well. Details for this procedure can be found in \cite{Rief}.

Now we can relate Rieffel's modules to Gabor analysis. Rieffel proved that the modules $V_L$ are finitely generated and projective, but he did not give examples of generating sets for them, nor did he give examples of projections in $A_L$ representing them. If $\mathcal{G}(g_1, \dots, g_N, L)$ is a multiwindow Gabor frame for $L$ where each $g_i \in \mathcal{S}(\R^d),$ then the $g_i$ clearly generate $V_L$ as an $A_L^\infty$-module. Furthermore, the element $Q \in M_N(A_L^\infty)$ defined by
$$Q_{ij}(z) = \langle g_i, \pi(z)\tilde{g_j}\rangle$$
is an idempotent\footnote{This idempotent will be a self-adjoint projection precisely when $\mathcal{G}(g_1, \dots, g_N, L)$ is a tight multiwindow Gabor frame.} representing $[V_L]$ in $K_0(A_L).$ This idempotent stores the coefficients of the windows $g_i$ with respect to the Gabor frame that they generate. These observations are due to Luef, who also constructed modules based on the modulation space $M^1(\R^d)$ \cite{Luef1}, \cite{Luef2}. In this case, we have the pre-$C^*$-algebra $l^1_\sigma(L)$ acting on the modulation space $M^1(\R^d)$ by time-frequency shifts. One can use this module in place of $V_L$ and all of Rieffel's results still hold. Using Gabor frames for quasicrystals, we will construct idempotents in $M_N(\mathcal{A}_\sigma)$ by adapting these ideas.

\end{subsection}

\begin{subsection}{Definition of $\mathcal{H}_\L$}
\label{hlambda}
Now we are ready to describe the module $\mathcal{H}_\L.$ Let $\L \subset \R^{2d}$ be a quasicrystal. Recall that $\sigma$ is the standard symplectic cocycle on $\R^{2d}.$ We construct a projective $\sigma$-representation of $R_\L$ by time-frequency shifts. Consider the (trivial) bundle of Hilbert spaces given by $\trans \times L^2(\R^d).$ Denote the fiber over a quasicrystal $T \in \trans$ by $H_T.$ An element $(T, T-z) \in R_\L$ acts as a map from $H_{T-z} \rightarrow H_{T}$ by
$$(T,T-z)f = \pi(z)f.$$
We could construct a module over $\mathcal{A}_\sigma$ by integrating this representation, however it would not have the correct topology to give a finitely generated projective module. 

Instead, we define a module over $L^1_\sigma(R_\L)$ which we will later complete to a module over $\mathcal{A}_\sigma.$ We begin with $C(\trans, M^1(\R^d)),$ the continuous functions on the transversal with values in $M^1(\R^d).$ Given $f \in L^1_\sigma(R_\L)$ and $\Psi \in C(\trans, M^1(\R^d))$ we define an action $I$ of $L^1_\sigma(R_\L)$ by
$$I(f)\Psi(T) = \sum_{z \in T} f(T, T-z)\pi(z) \Psi(T-z).$$
Since $\trans$ is compact, $||\Psi(T)||_{M^1} \leq C$ for some constant $C$ which is independent of $T.$ Thus the series converges in $M^1(\R^d).$ This representation is faithful since $L^1_\sigma(R_\L)$ is simple. We denote this $L^1_\sigma(R_\L)$-module by $\mathcal{H}_\L.$ We will denote by $\mathcal{C}_\L$ the linear subspace of $\mathcal{H}_\L$ of \textbf{transversally constant functions} which can be naturally identified with $M^1(\R^d).$ For $g \in M^1(\R^d)$ we denote by $\Psi_g \in \mathcal{C}_\L$ the function defined by $\Psi_g(T) = g.$ When $\mathcal{G}(g_1, \dots, g_N, \L)$ is a multiwindow Gabor frame we will show that $\Psi_{g_1}, \dots, \Psi_{g_N}$ generate $\mathcal{H}_\L$ as an $L^1_\sigma(R_\L)$-module and construct an associated idempotent in $L^1_\sigma(R_\L).$ 
\end{subsection}

\end{section}

\begin{section}{Gabor Frames for Quasicrystals}
\label{frameproofs}
\begin{subsection}{Existence of Multiwindow Gabor Frames for Quasicrystals}
\label{gaborframes}
Our first goal will be to prove Theorem \ref{frameexistence}. Given a quaisicrystal $\L,$ Theorem \ref{gocrthm} gives sufficient conditions for a single window Gabor frame to exist for $\L$ based on the size of $\rho(\L).$ To show that multiwindow frames exist, we first need the following lemma:

\begin{lemma}
\label{translates}
Suppose $\Lambda \subset \R^{2d}$ is FLC. Fix $\epsilon > 0.$ We can find finitely many disjoint translates $\{\Lambda_i\}_{i=1}^N$ so that $\bar{\Lambda} = \bigcup_{i=1}^N \Lambda_i$ has $\rho(\bar{\Lambda}) < \epsilon.$
\end{lemma}
\begin{proof}
First we let $R = \rho(\lambda) + \delta$ for some small $\delta.$ Since $\Lambda$ is FLC, there are only finitely many patterns of the form $B_R(z)\cap \Lambda$ up to translation. These patterns contain all the possible types of holes in $\Lambda,$ some of which have size larger than $\epsilon.$ Note that if we have a finite sequence $z_n$ then $\bigcup_{n=1}^N \Lambda + z_n$ is also FLC. Thus we can systematically shrink these holes one by one by taking unions of translates of $\Lambda.$ It will suffice to take a single pattern $P$ which contains a hole of size larger than $\epsilon$ and show how we can shrink that hole by a factor of 2. Repeating the procedure will shrink the hole below a size of $\epsilon.$

Choose a point $z \in P$ and let $c$ denote the center of the largest hole in $P.$ Then the set $\Lambda \cup (\Lambda -z + c)$ will no longer contain the patch $P.$ Instead, all occurrences of the patch $P$ in $\Lambda$ will now have a point in the center of the largest hole of $P,$ so that the largest hole will have been reduced in size by a factor of 2. 

This method does not ensure that the sets $\Lambda$ and $\Lambda -z+c$ will be disjoint, since the vector $z-c$ may lie in $\Lambda - \Lambda.$ To fix this, note that we do not need to place a point exactly in the center of the hole, but only very close to the center, in order to reduce the hole by a significant amount. Thus if $z-c \in \Lambda - \Lambda,$ we instead choose a point $c'$ close enough to $c$ so that $z-c' \notin \Lambda - \Lambda$ and the hole in $P$ is reduced by a factor of $2 - \eta$ for some small $\eta.$ We can find such a point $c'$ since $\Lambda$ FLC implies that $\Lambda - \Lambda$ is discrete. 

\end{proof}

\begin{proposition}
\label{multiframe}
Given a Delone set $\Lambda \subset \R^{2d}$ with FLC and $g \in M^1(\R^d),$ we can find a multiwindow Gabor frame for $\Lambda$ where the windows consist of time frequency translates of $g.$ Furthermore, this multiwindow Gabor frame will be an $M^p$-frame for all $p.$
\end{proposition}

\begin{proof}
Choose $\delta > 0$ so that $\omega_\delta(g) < 1.$ Applying Lemma \ref{translates}, we can find $\Lambda' = \bigcup_{i=1}^N (\Lambda + z_i)$ so that $\rho(\Lambda') < \delta.$ Then by Theorem \ref{frameexistence},
$$\mathcal{G}(g, \Lambda') = \bigcup_{i=1}^N \{ \pi(z+z_i)g \, |\, z \in \Lambda\}$$ 
is a Gabor frame and an $M^p$-frame for all $p.$ This is almost equal to the multiwindow Gabor system given by 
$$\bigcup_{i=1}^N\mathcal{G}(\pi(z_i)g, \Lambda) = \bigcup_{i=1}^N\{ \pi(z)\pi(z_i)g \, | \, z \in \Lambda \} = \bigcup_{i=1}^N \{ e^{-2\pi i x\omega_i}\pi(z+z_i)g \, |\, z \in \Lambda\}$$
where $z=(x,\omega)$ and $z_i=(x_i, \omega_i).$ The functions in the two Gabor systems differ only by phase factors, so $\bigcup_{i=1}^N\mathcal{G}(\pi(z_i)g, \Lambda)$ will satisfy the same frame inequalities as $\mathcal{G}(g, \Lambda')$ and thus $\bigcup_{i=1}^N\mathcal{G}(\pi(z_i)g, \Lambda)$ is a multiwindow Gabor frame with the same frame bounds (and $M^p$-frame bounds) as $\mathcal{G}(g, \Lambda').$

\end{proof}

\begin{corollary}
\label{cor}
If $g \in M^1(\R^d)$ and $\bigcup_{i=1}^N\mathcal{G}(\pi(z_i)g, \Lambda)$ is a multiwindow Gabor frame as constructed above, then so is $\bigcup_{i=1}^N\mathcal{G}(\pi(z_i)g, \Lambda')$ for any $\Lambda' \in \Omega_\Lambda.$
\end{corollary}

\begin{proof}
Since $\Lambda' \in \Omega_\Lambda,$ it contains exactly the same patches as $\Lambda.$ Thus the procedure in Lemma \ref{translates} also works to fill in the holes of $\Lambda',$ so that the sets $\bigcup_{i=1}^N \Lambda + z_i$ and $\bigcup_{i=1}^N \Lambda' + z_i$ have the same sized hole. Then the argument from Proposition \ref{multiframe} applies in exactly the same way to $\Lambda',$ showing that $\bigcup_{i=1}^N\mathcal{G}(\pi(z_i)g, \Lambda')$ is a multiwindow Gabor frame.
\end{proof}
\noindent Taken together, Proposition \ref{multiframe} and Corollary \ref{cor} immediately imply Theorem \ref{frameexistence}.

\end{subsection}

\begin{subsection}{Continuity and Covariance Properties of the Frame Operator}
Now we will investigate various continuity and covariance properties of the frame operator. When $\mathcal{G}(g_1, \dots, g_N, \L)$ is a multiwindow Gabor frame, we will denote the associated frame operator by $S^{\L}.$ Since we will never compare frame operators for Gabor frames with different windows, this notation is not ambiguous. We would like to understand the relationship between the frame operators $S^T$ and $S^{T'}$ when $T, T' \in \Omega_\Lambda.$ First we shall show that when $T' = T - z$ then there is a covariance condition relating $S^T$ and $S^{T'}.$
\begin{proposition}
\label{covariance}
If $\mathcal{G}(g_1, \dots, g_N, T)$ and $\mathcal{G}(g_1, \dots, g_N, T-w) $ are multiwindow Gabor systems for $T$ and $T-w$ respectively, then the frame operators $S^T$ and $S^{T-w}$ satisfy
$$S^T\pi(w) = \pi(w)S^{T-w}.$$

\end{proposition}

\begin{proof}
Fix $f \in L^2(\R^d).$ On the one hand we have
\begin{align*}
S^T\pi(w)f = \sum_{i=1}^N\sum_{z \in T} \langle \pi(w)f, \pi(z)g_i\rangle \pi(z) g_i 
&= \sum_{i=1}^N\sum_{z \in T} e^{2\pi i a\omega}\langle f, \pi(z-w)g_i \rangle \pi(z) g_i.
\end{align*}
where $z=(x,\omega)$ and $w=(a,b).$
On the other hand we have

\begin{align*}
\pi(w)S^{T-w}f &= \sum_{i=1}^N\sum_{z \in T}\langle f, \pi(z-w)g_i \rangle \pi(w)\pi(z-w)g_i 
=\sum_{i=1}^N\sum_{z \in T} e^{2\pi i a\omega}\langle f, \pi(z-w)g_i \rangle \pi(z) g_i 
\end{align*}
and so the two expressions are equal.

\end{proof}

We would also like to know something about the continuity of the frame operators over $\Omega_\Lambda.$ If $T_k \rightarrow T$ in $\Omega_\Lambda,$ we cannot expect $S^{T_k} \rightarrow S^T$ in the operator norm. However, we do have that $S^{T_k} \rightarrow S^T$ in the strong operator topology.

\begin{proposition}
\label{strongcontinuity}
Suppose $T_k \rightarrow T$ in $\Omega_\Lambda$ and the window functions $g_1, \dots, g_N$ lie in $M^1(\R^d).$ Then $S^{T_k} \rightarrow S^T$ in the strong operator topology on $B(M^1(\R^d)).$
\end{proposition}

\begin{proof}
Fix $f \in M^1(\R^d).$ Let $A = \text{max}\{||g_i||_{M^1}\}.$ Fix $\epsilon>0$ and choose a large cube $C$ so that for all $i$
$$\sum_{a \in \Z^n \setminus C} ||V_{g_i}f||_{L^\infty([0,1]^n+a)} < \frac{\epsilon}{4AN\text{rel}(\Lambda)}$$
where $N$ is the number of windows in the multiwindow frame.
Since $T_k \rightarrow T,$ we can choose $K$ so that for all $k \geq K, T_k$ agrees with $T$ on the cube $C$ up to a small translation, so that
$$\left\|\sum_{i=1}^N\sum_{z \in T \cap C} \langle f, \pi(z) g_i \rangle \pi(z)g_i - \sum_{i=1}^N\sum_{ z \in T_k \cap C} \langle f, \pi(z) g_i \rangle \pi(z) g_i\right\|_{M^1} < \frac{\epsilon}{2}.$$
Then for all $k \geq K$ we have
\begin{align*}
||S^{T}f - S^{T_k}f||_{M^1} &\leq 
\left\|\sum_{i=1}^N\sum_{z \in T \setminus C} \langle f, \pi(z) g_i \rangle \pi(z) g_i - \sum_{i=1}^N\sum_{z \in T_k \setminus C} \langle f, \pi(z) g_i \rangle \pi(z) g_i\right\|_{M^1} + \frac{\epsilon}{2}\\
&\leq \left(\sum_{i=1}^N\sum_{z \in T \setminus C} |\langle f, \pi(z) g_i \rangle| ||g_i||_{M^1} + \sum_{i=1}^N\sum_{z \in T_k \setminus C} |\langle f, \pi(z) g_i \rangle| ||g_i||_{M^1}\right) + \frac{\epsilon}{2}\\
&\leq A\left(\sum_{i=1}^N\sum_{z \in T \setminus C} |V_{g_i}f(z)| + \sum_{i=1}^N\sum_{z \in T_k \setminus C} |V_{g_i}f(z)|\right) + \frac{\epsilon}{2}\\
&\leq 2A\text{rel}(\Lambda)\left(\sum_{i=1}^N\sum_{a \in \Z^n \setminus C} ||V_{g_i}f||_{L^\infty([0,1]^n+a)}\right) + \frac{\epsilon}{2}\\
&< 2AN\text{rel}(\Lambda)\left(\frac{\epsilon}{4AN\text{rel}(\Lambda)}\right) + \frac{\epsilon}{2} = \epsilon.
\end{align*}
In the fourth inequality it is important to note that the inequality (\ref{relinequality}) holds not only for the norms, but also for the partial sums. The main reason this proof works is that $\text{rel}(T)$ is constant on $\Omega_\L.$ By applying inequality (\ref{relinequality}), this implies that we can find a cube $C$ so that the sum $S^Tf$ is arbitrarily small outside of $C$ independent of $T \in \Omega_\L.$
\end{proof}

Even though the mapping $T \rightarrow S^T$ will not be continuous when $B(M^1(\R^d))$ is given the norm topology, we can still show that all the frames $\mathcal{G}(g_1, \dots, g_N, T)$ have the same optimal frame bounds.

\begin{proposition}
\label{frameopbound}
Suppose $\mathcal{G}(g_1, \dots, g_N, T)$ is a frame for each $T \in \Omega_\L$ and each $g_i \in M^1(\R^d).$ For any $T \in \Omega_\L$ the optimal upper and lower frame bounds for $\mathcal{G}(g_1, \dots, g_N, T)$ are the same as those for $\mathcal{G}(g_1, \dots, g_N, \L).$ As a result, $||S^T||_{M^1} = ||S^\L||_{M^1}$ and $||(S^T)^{-1}||_{M^1}=||(S^\L)^{-1}||_{M^1}$ where $||\cdot||_{M^1}$ denotes the operator norm on $B(M^1(\R^d)).$
\end{proposition}

\begin{proof}
Let $A$ and $B$ denote the optimal lower and upper frame bounds for $\mathcal{G}(g_1, \dots, g_N, \L)$ so that for all $f \in M^1(\R^d)$
$$\sqrt{A}||f||_{M^1} \leq ||S^\L f||_{M^1} \leq \sqrt{B}||f||_{M^1}.$$
Since the translates of $\L$ are dense in $\Omega_\L,$ we can find a sequence of translates $\L - z_k \rightarrow T.$ Note that by Proposition \ref{covariance} the frame bounds are constant on the orbit of $\L.$ Since $\Lambda - z_k \rightarrow T, \, S^{\L-z_k} \rightarrow S^T$ in the strong topology by Proposition \ref{strongcontinuity}.

Now fix $f \in M^1(\R^d).$ W have $||S^{\L-z_k}f||_{M^1} \rightarrow ||S^Tf||_{M^1}.$ Since $\sqrt{A}||f||_{M^1} \leq ||S^{\L-z_k}f||_{M^1} \leq \sqrt{B}||f||_{M^1}$ for all $k,$ we have $\sqrt{A}||f||_{M^1} \leq ||S^Tf||_{M^1} \leq \sqrt{B}||f||_{M^1}.$ By reversing the roles of $T$ and $\L$ in this argument, we see that the upper and lower frame bounds for $T$ and $\L$ must be equal. The last remark follows since the lower and upper frame bounds are equal to $||(S^T)^{-1}||_{M^1}$ and $||S^T||_{M^1}$ respectively.

\end{proof}

\begin{corollary}
\label{cor2}
Suppose $g_1, \dots, g_N \in M^1(\R^d)$ and $\mathcal{G}(g_1, \dots, g_N, \Lambda)$ is an $M^1$-frame. Then for any $T \in \Omega_\Lambda, \mathcal{G}(g_1, \dots, g_N, T)$ is also an $M^1$-frame.
\end{corollary}

\begin{proof}
By examining the basic frame inequalities in Definition \ref{frameinequalities}, we can see that if $\mathcal{G}(g_1, \dots, g_N, \Lambda)$ is a frame then so is $\mathcal{G}(g_1, \dots, g_N, \Lambda - z)$ for any $z \in \R^{2d}.$ Given $T \in \Omega_\Lambda,$ we can find a sequence of translates $\L-z_k$ converging to $T.$ By Proposition \ref{strongcontinuity} we have $S^{\L-z_k} \rightarrow S^T$ in the strong topology on $B(M^1(\R^d)).$ The frame bounds for the frames $\mathcal{G}(g_1, \dots, g_N, \Lambda - z_k)$ are all equal by Proposition \ref{covariance}, so $S^T$ also satisfies those same frame bounds. In particular $S^T$ is bounded below on $M^1(\R^d),$ and thus $\mathcal{G}(g_1, \dots, g_N, T)$ is an $M^1$-frame. 
\end{proof}

\noindent Note the difference between Corollary \ref{cor} and Corollary \ref{cor2}. Corollary \ref{cor} says that there exist windows $\{g_i\}_{i=1}^N \subset M^1(\R^d)$ so that $\mathcal{G}(g_1, \dots, g_N, T)$ is a Gabor frame for any $T \in \Omega_\L.$ Corollary \ref{cor2} says that when $\mathcal{G}(g_1, \dots, g_N, \Lambda)$ is a multiwindow Gabor frame and each $g_i \in M^1(\R^d),$ then $\mathcal{G}(g_1, \dots, g_N, T)$ is \emph{automatically} also a Gabor frame for any $T \in \Omega_\L.$ The similarity between the Delone sets in $\Omega_\L$ is the key to Proposition \ref{strongcontinuity} which drives all of our results. This makes our proof of Corollary \ref{cor2} easier than the proof of Theorem 7.1 in \cite{GOCR}, which is the analogous result for more general point sets.

\end{subsection}

\begin{subsection}{Proof that $\mathcal{H}_\L$ is Finitely Generated and Projective}

Now we are ready to show that $\mathcal{H}_\Lambda$ is a finitely generated projective $L^1_\sigma(R_\L)$-module. Let $\Lambda \subset \R^{2d}$ be a quasicrystal, so that the associated algebra $L^1_\sigma(R_\L)$ can be described as in Section \ref{groupoidalgebra}. Fix $g_1, \dots, g_N \in M^1(\R^d)$ so that for any $T \in \Omega_\Lambda, \mathcal{G}(g_1, \dots, g_N, T)$ is a frame for $L^2(\R^d)$ and an $M^p$-frame for all $p.$ By Theorem \ref{frameexistence}, it is always possible to find functions satisfying this requirement. 

Now we can define two maps, which are generalizations of the analysis and synthesis maps for frames. We define the \textbf{noncommutative synthesis map} 
$$D: (L^1_\sigma(R_\L))^N \rightarrow \mathcal{H}_\L$$
by
$$D(\mathbf{1}_i) = \Psi_{g_i}$$
where $\mathbf{1}_i$ denotes the element of $(L^1_\sigma(R_\L))^N$ which is 0 except in the $i$th entry where it is equal to the identity element of $L^1_\sigma(R_\L).$ We extend this map to all of $(L^1_\sigma(R_\L))^N$ so that it is a continuous map of $L^1_\sigma(R_\L)$-modules, effectively by letting an element in $L^1_\sigma(R_\L)$ act on each $\Psi_{g_i}$ and then summing over $i.$

Denote by $\tilde{g_i}^T_{z} := (S^T)^{-1}\pi(z)g_i$ the $i$th dual frame element corresponding to $z \in T.$ We now define the \textbf{noncommutative analysis map} $C: \mathcal{H}_\Lambda \rightarrow (L^1_\sigma(R_\L))^N$ which sends a function $f \in \mathcal{H}_\L$ to
$$C(f) = (G_1, \dots, G_N) \in (L^1_\sigma(R_\L))^N$$
where 
$$G_i(T, T-z) := \langle f(T), \tilde{g_i}^T_{z} \rangle.$$
To see that $G_i \in L^1_\sigma(R_\L),$ we compute
\begin{align*}
\int_{\rpunc} |G_i|  
&= \int_{\trans} \sum_{z \in T} \left|\langle f(T), \tilde{g_i}^T_{z} \rangle\right| dT \\
&= \int_{\trans} \sum_{z \in T} \left|\langle (S^T)^{-1}f(T), \pi(z)g_i \rangle\right|dT\\
\end{align*}
which holds since $S^T$ is self-adjoint. For convenience we denote the function $(S^T)^{-1}f(T)$ by $F_T.$ Since $S^T$ is invertible in $B(M^1(\R^d))$ we have $F_T \in M^1(\R^d).$ Now we have

\begin{align*}
\int_{\trans} \sum_{z \in T} \left|\langle F_T, \pi(z)g_i \rangle\right|dT &= \int_{\trans} \sum_{z \in T} |V_{g_i}F_T(z)| dT\\
&\leq \text{rel}(\L) \, \int_{\trans} ||V_{g_i}F_T||_{W(L^\infty, L^1)}dT\\
&\leq C \, \text{rel}(\L) ||g_i||_{M^1} \, \int_{\trans} ||F_T||_{M^1}dT\\
&\leq C \, \text{rel}(\L) \, ||g_i||_{M^1} \int_{\trans} ||(S^T)^{-1}||_{M^1}||f(T)||_{M^1} dT < \infty.
\end{align*}
The inequality in the third line comes from Proposition $12.1.11$ in \cite{Gro}, and the constant $C$ is independent of $T.$ We see the integral is finite because $f \in C(\trans, M^1(\R^d))$ implies $||f(T)||_{M^1}$ is bounded on $\trans,$ and because Proposition \ref{frameopbound} shows that $||(S^T)^{-1}|| = ||(S^\L)^{-1}||$ for all $T.$ 
\begin{proposition}
The map $C$ is a map of $L^1_\sigma(R_\L)$-modules.
\end{proposition}

\begin{proof}
First note that the transversally constant functions $\mathcal{C}_\L$ are cyclic in $\mathcal{H}_\Lambda$ under the action of $L^1_\sigma(R_\L).$ For example, we can get all transversally locally constant functions by applying characteristic functions of the unit space of $\rpunc,$  and locally constant functions are dense in $C(\trans, M^1(\R^d)).$ Thus it will suffice to prove that $C$ is an $L^1_\sigma(R_\L)$-module map when $L^1_\sigma(R_\L)$ acts on $\mathcal{C}_\L.$

So assume that $f \in \mathcal{C}_\L$ and that $\eta \in L^1_\sigma(\rpunc).$ We will denote by $F \in M^1(\R^d)$ the function $f(T)$ which is independent of $T.$ On the one hand we have
\begin{align*}
C(I(\eta) f)_i(T,T-z) &= \left\langle \sum_{w \in T} \eta(T, T-w)\pi(w) F, \, \tilde{g_i}^T_{z} \right\rangle\\
&= \sum_{w \in T} \eta(T, T-w) \langle \pi(w) F, \tilde{g_i}^T_{z} \rangle \\
&= \sum_{w \in T} \eta(T, T-w) \langle F, T_{-a}M_{-b}\tilde{g_i}^T_{z} \rangle.
\end{align*}
where $w=(a,b).$ On the other hand we have
\begin{align*}
\eta * C(F)_i(T, T-z) = \sum_{w \in T} \eta(T, T-w) \langle F, e^{2\pi i a(b-\omega)}\tilde{g_i}^{T-w}_{z-w} \rangle.
\end{align*}
where $z=(x, \omega).$
We will show that 
$$T_{-a}M_{-b}\tilde{g_i}^T_{(x,\omega)} = e^{2\pi i a(b-\omega)}\tilde{g_i}^{T-w}_{z-w}.$$
Unpacking the definitions, we see that this is equivalent to showing
$$T_{-a}M_{-b}(S^T)^{-1}\pi(z)g_i = e^{2\pi i a(b-\omega)}(S^{T-w})^{-1} \pi(z-w) g_i$$
which is equivalent to
$$T_{-a}M_{-b}(S^T)^{-1}\pi(z) g_i = (S^{T-w})^{-1}T_{-a}M_{-b}\pi(z) g_i$$
after commuting $T_{-a}$ past $M_{(\omega - b)}$ on the RHS. We can cancel the $\pi(z)$ from both sides and simply show the operator equality
$$T_{-a}M_{-b}(S^T)^{-1} = (S^{T-w})^{-1}T_{-a}M_{-b}.$$
By inverting both sides we see this is equivalent to showing
$$S^T\pi(w) = \pi(w) S^{T-w}$$
which follows from Proposition \ref{covariance}.
\end{proof}

Now we can see that the maps $D$ and $C$ are well defined maps of $L^1_\sigma(R_\L)$-modules. Composing these maps, we get that 
$$DC f (T) = \sum_{i=1}^N \sum_{z \in T} \langle f(T), \tilde{g_i}^T_{z}\rangle g_i = f(T)$$
where the last equality holds since this is exactly the reconstruction formula for $f(T)$ using the Gabor frame $\mathcal{G}(g_1, \dots, g_N, T).$ Thus $C$ splits the map $D,$ showing that $\mathcal{H}_\Lambda$ is finitely generated (by the functions $\Psi_{g_i})$ and projective. Thus we have:
\begin{theorem}
$\mathcal{H}_\Lambda$ is finitely generated and projective as an $L^1_\sigma(R_\L)$-module.
\end{theorem}

If we compose these maps in the opposite order, we can construct an idempotent matrix $P \in M_N(L^1_\sigma(R_\L))$ which represents $\mathcal{H}_\Lambda$ in $K_0(\mathcal{A}_\sigma).$ We can write the elements of $P$ explicitly as functions in $L^1_\sigma(\rpunc)$ as 
$$P_{ij}(T, T-z) = \langle g_i, \tilde{g_j}^T_{z}\rangle.$$
\begin{remark}
Now that we have constructed an idempotent $P \in M_N(\mathcal{A}_\sigma)$ representing $[\mathcal{H}_\L]$ in $K_0(\mathcal{A}_\sigma),$ we can use this idempotent to endow $\mathcal{H}_\L$ with the structure of a Hilbert pre-$C^*$-module. While this construction is purely formal, it is an interesting question whether the structure can be described using the language of Gabor analysis.
\end{remark}
To compute the trace of this idempotent (and thus the dimension of the module $\mathcal{H}_\L$) we apply the normalized trace on $M_N(\mathcal{A}_\sigma)$ to get
$$Tr(P) = \frac{1}{N} \sum_{i=1}^N \int_{\trans} \langle g_i, \tilde{g_i}^T_{0}\rangle dT.$$
By applying Birkhoff's Ergodic Theorem, the integrals can be replaced by averages over the orbits of $\Lambda.$ Thus we get
\begin{align*}
Tr(P) = \lim_{k \to \infty}\frac{1}{N |\Lambda \cap C_k|} \sum_{i=1}^N \sum_{z \in (\Lambda \cap C_k)} \langle g_i, \tilde{g_i}^{\Lambda-z}_{(0,0)}\rangle
\end{align*}
where $C_k$ is the cube centered at the origin with side length $k.$ We would like to rewrite this sum so that it involves only the dual frame for $\mathcal{G}(g_1, \dots, g_N, \Lambda).$ We can use Proposition \ref{covariance} to rewrite $\tilde{g_i}^{\Lambda-z}_{(0,0)}$ as
$$\tilde{g_i}^{\Lambda-z}_{(0,0)} = (S^{\Lambda-z})^{-1}g_i = (S^{\Lambda-z})^{-1}T_{-x}M_{-\omega} M_\omega T_x g_i = T_{-x}M_{-\omega} (S^\Lambda)^{-1} \pi(z) g_i.$$
Now we can rewrite the sum as
\begin{align*}Tr(P) &= \lim_{k \to \infty}\frac{1}{N|\Lambda \cap C_k|} \sum_{i=1}^N  \sum_{z \in (\Lambda \cap C_k)} \langle g_i,T_{-x}M_{-\omega} (S^\Lambda)^{-1} \pi(z) g_i \rangle \\
&= \lim_{k \to \infty}\frac{1}{N|\Lambda \cap C_k|}\sum_{i=1}^N \sum_{z \in (\Lambda \cap C_k)} \langle \pi(z) g_i, \tilde{g_i}^\Lambda_{z}\rangle
\end{align*}
which involves only the Gabor frame $\mathcal{G}(g_1, \dots, g_N, \Lambda)$ and its dual. These averages coincide precisely with the \textbf{frame measure} introduced in \cite{BCHL2}. In Theorem 4.2 (b) they show that for a single window frame, the averages above are equal to $\frac{1}{\text{Dens}(\Lambda)}.$ Their results are easily generalized to show that this also holds for multiwindow frames, so we get the following result:
\begin{corollary}
The dimension of $\mathcal{H}_\Lambda$ is equal to $\frac{1}{Dens(\Lambda)}.$
\end{corollary}
\noindent Thus we have completed the proof of Theorem \ref{fingenproj}. Note that the realization of the frame measure as the dimension of a projective module gives a structural reason why it should be invariant of the choice of windows for the frame.
\end{subsection}

\end{section}

\begin{section}{$K$-theory and Twisted Gap Labeling}
\label{KGap}

Now we will look at the simpler case when $\L$ is a marked lattice and investigate the way that $\mathcal{H}_\L$ fits into $K_0(\mathcal{A}_\sigma).$ A \textbf{marked lattice} is a lattice $L \subset \R^d$ where each point $l \in L$ is also assigned a color. For simplicity we shall always assume $L=\Z^d,$ and the arguments given can be easily adapted to apply when $L$ is a general lattice. We can construct the hull $\Omega_\L$ in exactly the same way when $\L$ is a marked lattice. As point sets, all elements of $\Omega_\L$ will be a translate of the integer lattice, however the sets are only considered equal when their colorings are also the same. We will always assume that a marked lattice has an aperiodic coloring with FLC and UCF.

\begin{example}
The \textbf{chair tiling} is a substitution tiling of $\R^2$ where the vertices of the tiles are contained in $\Z^2$ (see \cite{Sadun}). We denote the set of vertices of the tiling by $V.$ We can take $\Z^2$ and color the points in $V$ red and all other points blue. This is an example of a marked lattice whose hull has the same properties as the hull of a quasicrystal.
\end{example}

When $\L$ is a marked lattice, the hull $\Omega_\L$ has the structure of a fiber bundle $\trans \rightarrow \Omega_\L \rightarrow \T^d.$ It is the suspension of $\trans$ by an action of $\Z^d.$ We can understand the second map using the associated $C^*$-algebras. Denote by $\mathcal{A}:= C^*(R_\L)$ the untwisted groupoid $C^*$-algebra of $R_\L$ which in this case is isomorphic to the crossed product $C(\trans) \rtimes \Z^d.$ Then we have a map
$$i: C(\T^d) \cong \C^*_r(\Z^d) \rightarrow \mathcal{A}$$
where $i$ takes a function $f \in C_0(\Z^d)$ and extends it to a function $F$ on $R_\L$ by making it constant in the direction of $\trans,$ i.e. $F(T-z, T) = f(z).$ In other words, the image of $i$ is generated by the functions on $R_\L$ which do not depend on the colorings of the points. The map $i$ is the discrete analog of the map $C(\T^d) \rightarrow C(\Omega_\L)$ induced by the fibration. Similarly when we twist by a standard cocycle $\theta$ we have an induced map 
$$j: A_\theta \rightarrow \mathcal{A}_\theta$$
from a noncommutative torus into $\mathcal{A}_\theta.$ Both $i$ and $j$ preserve the trace on $C^*_r(\Z^d)$ and $A_\theta$ respectively. 

Our goal is to prove that the induced maps $i_*$ and $j_*$ are injective on $K_0.$ When $\theta$ is totally irrational, the trace on $A_\theta$ is injective. Since $j$ preserves the trace this immediately implies that $j_*$ will be injective. We will show that the maps $i_*$ and $j_*$ are compatible, so that the injectivity of $j_*$ for a totally irrational cocycle implies the injectivity of $i_*.$ 

\begin{proposition}
\label{commutes}
Let $\L = \Z^d$ be a marked lattice. Fix a cocycle $\theta_1$ on $\Z^d,$ and let the maps $i$ and $j$ be defined as above. Also fix a homotopy $\theta$ between $\theta_1 = \theta(1)$ and the trivial cocycle $\theta(0).$ Then we have a commutative diagram
\begin{center}
$\begin{CD}
K_0(C^*_r(\Z^d)) @>i_*>> K_0(\mathcal{A})\\
@V\cong VV                @VV\cong V \\
K_0(C^*_r(\Z^d \times [0,1], \theta)) @>k_*>> K_0(C^*_r(\Z^d \rtimes \, \trans \times [0,1], \theta)) \\
@V\cong VV                @VV\cong V \\
K_0(A_{\theta_1})    @>>j_*> K_0(\mathcal{A}_{\theta_1})
\end{CD}$
\end{center}
where the vertical arrows come from the isomorphisms in Theorem \ref{Gillaspy} and the second horizontal map is induced by the map $k : C^*_r(\Z^d \times [0,1], \theta) \rightarrow C^*_r(\Z^d \rtimes \, \trans \times [0,1], \theta)$ given by $i$ on the fiber at $0$ and the map $j_t : A_{\theta(t)} \rightarrow \mathcal{A}_{\theta(t)}$ on the fiber at $0 < t \leq 1.$ 
\end{proposition}

\begin{proof}
We will prove only the commutativity of the upper square; commutativity of the lower square follows by a similar argument. Choose a projection $P \in M_N(C^*_r(\Z^d)).$ We can lift this to a path of projections $P_t,$ yielding an element of $K_0(C^*_r(\Z^d \times [0,1], \theta)).$ When we map this via $k_*,$ we simply extend the projection on each fiber by making it constant in the $\trans$ direction. Following the maps the other way around, we can take $P$ and extend it to be constant in the $\trans$ direction, then lift it to a path of projections. It is clear that $k_*(P_t)$ is one such possible lift, so we are done.
\end{proof}

\begin{theorem}
Let $\L=\Z^d$ be a marked lattice and fix any cocycle $\theta$ on $\Z^d.$ Then the maps $i_*$ and $j_*$ are injective. We can compare their images with the image of the canonical map $r_*:K_0(C(\trans)) \rightarrow K_0(\mathcal{A}_\theta)$ and we find that the intersection is generated by $[\bf{1}],$ the class of the rank $1$ trivial module. 
\end{theorem}

\begin{remark}
Note that this immediately implies Theorem \ref{injective}, since the map $i_*$ is just the noncommutative version of the map $p^*.$
\end{remark}

\begin{proof}
First note that when $\theta$ is totally irrational, the map $Tr_* \circ j_*$ is injective, so $j_*$ is injective as well. Thus by Proposition \ref{commutes}, we see that $i_*$ must also be injective. Now let $\theta$ be any cocycle. Since $i_*$ is injective, by Proposition \ref{commutes} we see that $j_*$ must be as well.

Now we compare the images of $i_*$ and $j_*$ with the image of $r_*.$ First suppose $\theta$ is a totally irrational cocycle, and that the intersection of the groups $Tr_{A_\theta *}(K_0(A_\theta))$ and $Tr_*(K_0(\mathcal{A}))$ is equal to $\Z \subset \R.$ This is possible since $Tr_*(K_0(\mathcal{A}))$ is countable, so we can simply choose the entries in the matrix for $\theta$ to be rationally independent from $Tr_*(K_0(\mathcal{A})).$ Now it is clear that the image of $j_*$ is disjoint from the projections in $C(\trans)$ (except for multiples of the identity) since this is true after applying the trace. Now note that projections in $C(\trans)$ are preserved by the vertical isomorphisms on the RHS of the diagram in Proposition \ref{commutes}, so the same must be true for $i_*.$ Finally, using the diagram from Proposition \ref{commutes}, the theorem holds when $\theta$ is an arbitrary cocycle.
\end{proof}

We can interpret the results above in terms of the modules $\mathcal{H}_\L.$ When $\L$ is a marked lattice, a Gabor frame for $\L$ is simply a lattice Gabor frame and does not depend at all on the colorings of the points in $\L.$ Furthermore, when $\L=\Z^{2d}$ as a point set then the standard symplectic cocycle $\sigma|_\L$ is the trivial cocycle. In this case, we can use the construction of $V_\L$ in Section \ref{nctori} to get a module over $C^*_r(\Z^{2d}),$ and $i_*([V_\L]) = [\mathcal{H}_\L].$ To construct modules over $\mathcal{A}_\theta$ for general $\theta,$ we follow Rieffel's construction and apply a linear map $A$ to $\L$ with $A^*\sigma = \theta.$ Then we get a module $V_\L$ over the noncommutative torus $A_{A\L}$ and $j_*([V_{A\L}]) = [\mathcal{H}_{A\L}].$ Thus our modules precisely describe the images of $i_*$ and $j_*$ for even dimensional $\L.$ With a little more work, it seems likely that Rieffel's more general method can be adapted to construct modules when $\L$ is odd dimensional as well.

Now we will describe these results in dimension $2,$ where they allow us to determine the entire gap labeling group. Note that any cocycle $\theta$ on $\Z^2$ is determined by a single real number (also denoted $\theta),$ which is the only non-zero entry in the associated skew symmetric matrix. When $\L=\Z^2$ is a marked lattice, we can compute its $K$-theory by applying the Pimsner-Voiculescu exact sequence twice, or by applying the associated Pimsner-Voiculescu spectral sequence \cite{Kasp}, \cite{Elst}. In this case we have
$$K_0(C(\trans) \rtimes \Z^2) = C(\trans, \Z)_{\Z^2}\oplus \Z$$
where $C(\trans, \Z)_{\Z^2}$ denotes the group of coinvariants of the action of $\Z^2$ on $\trans.$ Here the extra copy of $\Z$ comes from the inclusion 
$$K^0(\T^2) \cong K_0(C^*_r(\Z^2)) \rightarrow K_0(C(\trans) \rtimes \Z^2)$$
of the group algebra of $\Z^2$ into $C(\trans) \rtimes \Z^2,$ and the summand $C(\trans, \Z)_{\Z^2}$ comes from the inclusion
$$K_0(C(\trans)) \rightarrow K_0(C(\trans) \rtimes \Z^2).$$
The extra generator is precisely the image of the Bott vector bundle in $K^0(\T^2).$ Thus from our results above, we immediately have

\begin{proposition}
When $\L=\Z^2$ is a marked lattice, the gap labeling group of $\mathcal{A}_\theta$ is 
$$Tr_*(K_0(\mathcal{A}_\theta)) = \mu(C(\trans,\Z)) + \theta \Z.$$
\end{proposition}

We can also determine the gap labeling group when we have a quasicrystal $\L \subset \Z^2.$ In this case, we can construct a marked lattice $\Gamma=\Z^2$ by coloring the points of $\L$ red and the remaining points blue. Then $\trans^\L$ sits as a clopen set in $\trans^{\Gamma}$ with measure equal to $\text{Dens}(\L).$ This shows that the gap labeling group of $\mathcal{A}^\L_\theta$ is equal to $\frac{1}{\text{Dens}(\L)} Tr_*(K_0(\mathcal{A}_\theta^{\Gamma})),$ which is in turn equal to $\mu(C(\trans^\L), \Z) + \frac{\theta}{\text{Dens}(\L)}\Z.$ Thus we have:

\begin{theorem}
When $\L \subset \Z^2$ is a quasicrystal, the gap labeling group of $\mathcal{A}_\theta$ is  
$$Tr_*(K_0(\mathcal{A}_\theta)) = \mu(C(\trans,\Z)) + \frac{\theta}{\text{Dens}(\L)} \Z.$$
\end{theorem}

\noindent Note that in dimension two a matrix $A$ satisfies $A^*\sigma = \theta$ exactly when $\det(A)=\theta.$ Thus the module $\mathcal{H}_{A\L}$ has trace $\frac{1}{\text{Dens}(A\L)} = \frac{\theta}{\text{Dens}(\L)}$ and represents the extra generator in $K_0(\mathcal{A}_\theta).$

\begin{section}{Concluding Remarks}
In the previous section we have computed the twisted gap labeling group for two dimensional quasicrystals which are subsets of lattices. The gap labeling group will always contain the group of patch frequencies $\mu(C(\trans, \Z)),$ however we have shown that after twisting by a cocycle the gap labeling group can become larger. If we fix a cocycle $\theta$ on $\Z^2$ then the extra summand in $K_0(\mathcal{A}_\theta)$ is represented by modules of the form $\mathcal{H}_{A\L}$ where $A$ is a linear map with determinant equal to $\theta.$

By results of Sadun and Williams \cite{SadWil}, given any quasicrystal $\L$ it is possible to find a homeomorphism between $\Omega_\L$ and $\Omega_{\Gamma},$ where $\Gamma$ is a marked lattice. Thus one might be tempted to think that our results can be used to compute the gap labeling group for any quasicrystal twisted by a standard cocycle. In Sadun and Williams' construction, they take a quasicrystal and systematically deform it to a lattice, using the colorings of points to keep track of the local structure of $\L.$ However, if we take a standard cocycle on a quasicrystal $\L$ and see what happens to the cocycle after we deform $\L$ to a marked lattice, we are quite likely to get a \emph{non-standard} cocycle on the marked lattice. Thus our results do not apply in the general case.

Nonetheless, there is some hope in using the strategy employed here to find the gap labeling group in the general case, even in higher dimensions. In the case of a marked lattice, we applied linear maps to $\L$ to construct modules over $\mathcal{A}_\theta$ for all standard cocycles $\theta.$ Linear maps are a simplistic deformation and ignore the aperiodic coloring of $\L.$ One could instead apply a more general deformation of $\L,$ so long as it preserves the groupoid $R_\L.$ Here we are thinking specifically of the types of deformations described by Kellendonk in \cite{Kel}. Given a deformation $D$ of a quasicrystal $\L$ and a cocycle $\theta$ on $R_\L,$ it may be profitable to investigate conditions under which $\sigma|_{D\L} = \theta.$ Such a deformation immediately leads to a module $\mathcal{H}_{D\L}$ over $\mathcal{A}_\theta$ whose trace is equal to $\frac{1}{\text{Dens}(D\L)}.$ Thus there is some hope to use our methods along with deformations of quasicrystals to compute the twisted gap labeling group in much greater generality. 

\end{section}

\end{section}

\bibliographystyle{plain}
\bibliography{bibtex}

\end{document}